\documentclass[12pt]{amsart}
\usepackage{amssymb}
\usepackage{float}
\usepackage[all]{xy}

\input{xiamacro2.sty}

\title[Cohomology of representation varieties]{The algebraic de Rham cohomology of representation varieties}

\author{Eugene Z. Xia}
\address{Department of Mathematics, National Cheng Kung University
and National Center for Theoretical Sciences, Tainan 70101}
\email{ezxia@ncku.edu.tw}

\thanks{The author gratefully acknowledges partial support by the Ministry of Science and Technology, Taiwan with grants.}

\usepackage[colorlinks]{hyperref}
\renewcommand{\o}{\operatorname}
\renewcommand{\b}{\bullet}
\renewcommand{\d}{\partial}
\begin{document}

\begin{abstract}
The $\SL(2,\C)$-representation varieties of punctured surfaces form natural families parameterized by holonomies at the punctures.  In this paper, we first compute the loci where these varieties are singular for the cases of one-holed and two-holed tori and the four-holed sphere.  We then compute the de Rham cohomologies of these varieties of the one-holed torus and the four-holed sphere when the varieties are smooth via the Grothendieck theorem.  Furthermore, we produce the explicit Gau\ss-Manin connection on the natural family of the smooth $\SL(2,\C)$-representation variety of the one-holed torus.
\end{abstract}

\maketitle
\section{Prelude}

Let $\Sigma = \Sigma_{g,m}$ be a compact oriented surface of genus $g$
with $m$ punctures
$\CC = \{\CC_1,...,\CC_m\}.$
Denote by $\pi = \pi_1(\Sigma)$ its fundamental group.  Let \(G\) be a reductive complex algebraic group and $\Hom(\pi, G)$ the space of homomorphisms (representations) from $\pi$ to $G$.  $\Hom(\pi, G)$ inherits a variety structure from $G$ and
$G$ acts on $\Hom(\pi,G)$ by equivalence of representations (conjugation)
$$
G \times \Hom(\pi,G) \lto \Hom(\pi,G), \ \ \  (g,\rho) \mapsto g \rho g^{-1}.
$$
Denote by
$$
\M(G) = \Hom(\pi,G)/G
$$
the categorical quotient of equivalent representations.
Fix a conjugacy class $C_i \subset G$ for each puncture
$\CC_i$ and let $C = \{C_1, \cdots, C_m\}$. Let
$$
\Hom_C(\pi, G) = \{\rho \in \Hom(\pi, G): \rho(\CC_i) \in
C_i, \ \  \mbox{for} \ \  1 \le i \le m \}.
$$
The $G$-action preserves $\Hom_C(\pi,G)$ and the
representation variety is the categorical quotient
$$
\M_C(G) = \Hom_C(\pi, G)/G.
$$

The representation variety $\M_C(G)$ is of great interest because it is the (coarse) moduli space of integrable $G$-connections on $\Sigma_{g,m}$ \cite{De1, DK1}.     Fix $G = \SL(2,\C)$ and let
$$\M = \M(\SL(2,\C)), \ \ \ \M_C = \M_C(\SL(2,\C)).$$
As the conjugacy classes in $C$ vary in $G$, the moduli spaces $\M_C$ vary.  That is
the varieties $\M_C$ form a family parameterized by $C$.
In this paper, we first identify the singular loci for the cases of $\Sigma_{1,1}, \Sigma_{1,2}$ and $\Sigma_{0,4}$.

In the cases of $\Sigma_{1,1}$ and $\Sigma_{0,4}$, the $\M_C$'s are 2-dimensional.  In these two cases, their homologies have been calculated via topological methods \cite{GN1}.
A remarkable theorem of Grothendieck \cite{Gr0} states that the hypercohomology of the algebraic de Rham complex of a smooth variety computes its smooth de Rham cohomology.  This provides an algebraic method for computing $\H_{dR}^*(\M_C)$.
We then carry out the computations for the representation varieties for these two cases.  
There are pure algorithmic approaches to these problems; however, these general methods tend to overwhelm computers \cite{OT1, Sc1, Wa1}.  We compute our results directly while taking advantages of the computer resources available, especially {\em Macaulay2} \cite{GS1}.

These families have natural integrable connections, the Gau\ss-Manin connections \cite{KO1}.  We compute this connection explicitly for the family of representation varieties of $\Sigma_{1,1}$.
\vskip 0.2in
\centerline{Acknowledgement}

Many results contained in this paper are obtained with the aid of {\em Mathematica} \cite{Wr1}, {\em Singular} \cite{DGPS} and especially {\em Macaulay2} \cite{GS1}.  The author also benefited from discussions with William M. Goldman and Jiu-Kang Yu.  The latter also provided the author with a simple yet effective package to compute differential forms.

\section{Generalities}\label{sec:gen}
\subsection{Smooth varieties and their cohomologies}
Let $X$ be a smooth algebraic variety over $\C$ with structure sheaf $\Os_X$.
Denote by $(\Omega_X^\b, d)$ the algebraic de Rham complex of $X$:
\begin{equation}\label{seq:de Rham}
\xymatrix{
(\Omega_X^\b, d): & \Omega_X^0 \ar[r]^{d_0} & \Omega_X^1 \ar[r]^{d_1} & \cdots
}
\end{equation}
A remarkable theorem of Grothendieck states \cite{Gr0}:
\begin{thm}[Grothendieck]\label{thm:Groth}
The (hyper)cohomologies of $(\Omega_X^\b, d)$ coincide with the smooth de Rham cohomologies of $X$.
\end{thm}
\subsection{Relative de Rham complex and cohomologies}
Let $Y \to \Spec(\C)$ be a smooth $\C$-variety and $f: X \to Y$ a smooth $Y$-variety.
Denote by $f^* : \Os_Y \lto \Os_X$ the corresponding morphism between the structure sheaves.  From these come the three de Rham complexes [\S 2.8, \cite{Ha1}]:
\begin{equation}\label{obj:de Rham}
(\Omega_X^\b, d), (\Omega_Y^\b, d), (\Omega_{X/Y}^\b, d).
\end{equation}
Each complex is associated with their respective (hyper-)cohomologies.  The relative de Rham cohomologies associated with $(\Omega_{X/Y}^\b, d)$ are cohomologies of $\Os_Y$-sheaves
$$
\mH^i := \mH^i(X) := R^i f_*(\Omega_{X/Y}^\b).
$$
Let $\Phi : S \lto Y$ be a flat morphism.  Then by base extension, we obtain the $S$-scheme $\Phi^*(X)$ and its associated de Rham complex.
\begin{prop}\label{prop:localization}
$$
\Phi^*(\mH^i) \stackrel{\cong}{\lto} R^i (f \circ \Phi)_*(\Omega_{\Phi^*(X)/S}).
$$
\end{prop}
\begin{proof}
See [Proposition 5.2, \cite{Ha1}] .
\end{proof}
In particular, this is true for localization at a maximal ideal $P \subseteq \Os_Y$, i.e. for $\Phi : Y_P \lto Y$.

Let $\C = k(P)$ be the residue field at the closed point $P$ and
$$\phi : \Spec(k(P)) \lto Y.$$  Then we obtain the $S$-scheme $U = \phi^*(X)$ by base extension and the associated de Rham complex $(\Omega_U^\bullet, d)$.  Denote by $\H^\bullet$ the de Rham cohomologies of $U$.

\subsection{The Gau\ss-Manin connection}
Assume $f$ to be smooth.  Then there is an exact sequence:
\[\xymatrix{
0 \ar[r] & \Os_X \otimes_{f^*} \Omega_Y^1 \ar[r] & \Omega_X^1 \ar[r] & \Omega_{X/Y}^1 \ar[r] & 0.
}\]
This gives rise to a filtration $F$ on $\Omega_X^\bullet$
$$
F^i = \im(\Omega_X^{\bullet - i} \otimes_{f^*} \Omega_Y^i \stackrel{\wedge}{\lto} \Omega_X^i).
$$
The resulting spectral sequence has its $(E_1, d_1)$ pair such that \cite{KO1}
$$
E_1^{p,q} \cong \Omega_Y^p \otimes_{f^*} \mH^q.
$$
The Gau\ss-Manin connection on $\mH^q$ is the differential $\nb = d_1^{0,q}$ in the following complex \cite{KO1}
\[\xymatrix{
0 \ar[r] & \mH^q \ar[r]^{d_1^{0,q}} & \Omega_Y^1 \otimes_{f^*} \mH^q \ar[r]^{d_1^{1,q}} & \Omega_Y^2 \otimes_{f^*} \mH^q \ar[r] & \cdots.
}\]

\section{Singular and smooth varieties}
For the rest of the paper, unless otherwise specified, we assume all varieties are affine over $\C$.

\subsection{Rings, modules and affine varieties}
Denote by $\x$ the set $\{x_1, x_2, \cdots x_n\}$ and $\Os \cong \C[\x]$, the coordinate ring of $\C^n$ and
by $(\Omega^\bullet, d)$ the algebraic de Rham complex over $\Spec(\C[\x])$.
Let $a \in \Z_{\ge 0}^n$.  We will use the standard notations
$$
|a| := \sum_{i=1}^n a_i, \ \ \ \ \x^a := \prod_{i=1}^n x_i^{a_i}.
$$
Let $N = \{j : 1 \le j \le n\}$ be the ordered index set of $n$ elements.  For an ordered subset $K \subseteq N$, write $dx_K$ for $\wedge_{j \in K} dx_j$.
Then $\Omega^i$ is generated as an $\Os$-module by $\{dx_K : |K| = i\}.$
\begin{defin}
Let $w \in \Omega^i$.
\begin{enumerate}
\item Denote by $\d_j$ the differential operator $\frac{\d}{\d x_j}$.
\item $w$ is a monomial form of degree $|a|$ if $w = \x^a dx_K$.
\item $\deg(w)$ denotes the maximum degree of the monomial forms in $w$.
\end{enumerate}
\end{defin}
Let
$$I = (\phi_i : 1 \le i \le k) \subseteq \Os$$
be the (finitely generated) ideal of definition of $U$, which is to say that  $U = \Spec(\Os_U)$ with
\begin{equation}\label{seq:definition}
\xymatrix{
0 \ar[r] & I \ar[r] & \Os \ar[r] & \Os_U \ar[r] & 0.
}
\end{equation}
This induces an inclusion $\iota : U \to \C^n$.
\begin{rem}
For any module $M$, we shall always use elements of $M$ to denote elements of quotients of $M$ when contexts are clear.
\end{rem}
\begin{rem}\label{rem:duality}
The map
$$
P : \Omega^0 \lto \Omega^n, \ \ \ f \mapsto f dx_N
$$
is an $\Os$-module isomorphism.
\end{rem}

\subsection{Gr\"obner basis and singularity}
Let $F$ be a free $\Os$-module and $W$ a complete order on the monomials of $F$ [\S 15, \cite{Ei1}].  For any $v, w \in F$, we write $W(v) > W(w)$ if the leading monomial of $v$ is greater than the leading monomial of $w$ according to the order $W$.
A monomial order corresponds to a weight matrix and we do not distinguish the two [\S 2.4, \cite{CLO1}].
\begin{defin} [\S 16, \cite{Ei1}]
Let $\JJ \subseteq \Os$ be the ideal generated by the $c \times c$ minors of the Jacobian $[\d_j \phi_i]$, where $1 \le j \le n$ and $c$ is the codimension of $I$ ($U)$.  The Jacobian ideal of $I$ is
$$
J(I) :=  I + \JJ \subseteq \Os.
$$
\end{defin}
\begin{prop}[\S 16, \cite{Ei1}] \label{prop:singular}
$U$ is smooth if and only if $J(I) = \Os$.
\end{prop}
Hence one may determine whether $U$ is smooth by computing a Gr\"obner basis $J_G$ for $J(I)$.

\begin{rem}\label{rem:degree}
Let $w$ be a monomial form.  Then either $dw=0$ or $\deg(w) = \deg(d w) + 1$.
\end{rem}

\begin{defin}\label{def:compatible weight}
A monomial order $W$ on $\Omega^i$ is degree-modified if $\deg(\eta_2) > \deg(\eta_1)$ implies $W(\eta_2) > W(\eta_1)$.

If $W$ is a monomial order on $\Omega^0$, then $W$ induces a monomial order on $\Omega^{n}$ and vice versa, via $P$:
$$W(\x^a) \longleftrightarrow W(P(\x^a)).$$  With respect to $P$, our order will always satisfy: $W(w) > W(v)$ if and only if $W(P(w)) > W(P(v))$ (See Remark~\ref{rem:duality}).
\end{defin}

\section{Computing algebraic de Rham Cohomology}

Designing effective algorithms to compute algebraic de Rham cohomologies for smooth Noetherian varieties is an interesting problem.  There is a general algorithm for smooth projective varieties \cite{Wa1}.  For the affine case, there is a general algorithm to compute the upper bounds of the Betti numbers \cite{Sc1}.  As typical with these methods, they depend on the non-commutative Gr\"obner basis computation and the computational complexity is often large.
This section describes how to explicitly compute the top algebraic de Rham cohomology ($\H^{\dim(U)}$) of a smooth affine variety corresponding to a principal ideal domain.

We begin by recalling the inclusion morphism $\iota : U \lto \C^{n}$.  The first thing to notice is that coherent sheaves on affine varieties are acyclic.  This implies that hypercohomology reduces to cohomology of complex:
\begin{cor}\label{cor:hyper}
Suppose $U$ is affine.
Then the hypercohomologies of the algebraic de Rham complex $(\Omega_U^\b, d)$ are
$$
\H^i := \H^i(U) := \frac{\ker(d_i)}{\im(d_{i-1})}.
$$
Define $h^i = \dim(\H^i)$.
\end{cor}

\subsection{Algebraic de Rham cohomology}
Again let $I \subseteq \Os$ be the ideal of definition of $U$ and let $l = \dim(U)$.  From Sequence~(\ref{seq:definition}), we obtain an exact sequence of $\Os$-modules where $Q$ is the quotient
\[\xymatrix{
d_0 I \wedge  \Omega^{\b-1} \ar[r] &  \Omega^\b \ar[r]^{\text{proj}} & Q^\b \ar[r] & 0.
}\]
Pulling back this sequence by $\iota$ we have
$$\Omega_U^\b = \iota^*(Q^\b) = \Os_U \otimes_{\iota^*} Q^\bullet$$
which is both an $\Os$- and an $\Os_U$-module.  We obtain the exact sequence of $\Os$-modules
\begin{equation}\label{seq:O-module}
\xymatrix{
d_0 I \wedge \Omega^{\b-1} +  I \Omega^\b \ar[r] & \Omega^\b \ar[r] & \Omega_U^\b \ar[r] & 0.
}
\end{equation}

\begin{defin}
For two $i$-forms $w_1, w_2 $, write $w_1 \sim w_2$ (cohomologous) if $w_1 = w_2 + d u$ for some $(i-1)$-form $u.$
\end{defin}
\begin{rem}\label{rem:acyclic}
$\C^n$ is (de Rham) acyclic.
\end{rem}
\begin{rem}
$d$ is not $\Os$-linear, so it is important to distinguish $\C$-linear morphisms and $\Os$-morphisms.
\end{rem}

\subsection{The top cohomology}
Assume $U$ to be smooth of dimension $l$ for the rest of this section.
Then $\Omega_U^{l+1} = 0$ and every form in $\Omega_U^l$ is closed.
Hence we have the following $\C$-linear commutative diagram with exact rows:
\begin{figure}[H]
\begin{equation*}
\xymatrix{
d\Omega_U^{l-1}\ar@{^{(}->}[r] & \Omega_U^l \ar[r] & \H^l \ar@{=}[d] \ar[r] & 0\\
d\Omega^{l-1} + I \Omega^l  + dI \wedge \Omega^{l-1} \ar@{^{(}->}[r]  \ar[u] & \Omega^l \ar[r]^p
 \ar[u] & \H^l \ar[r] & 0\\
}
\end{equation*}
\caption{}
\label{fig:3}
\end{figure}
The up arrows are projections.  This means that we obtain a rather simple set of generators (compare \cite{Sc1, Sc2}):
\begin{prop}
The cohomology $\H^l$ is generated by the monomials
$$\{p(\x^a dx_K) : a \in \Z_{\ge 0}^n, \ \ K \subseteq N, \ \ |K| = l\}.$$
\end{prop}

\begin{lem}\label{lem:small}
For any $j$,
$$d(dI \wedge \Omega^{j-1}) = dI \wedge d\Omega^{j-1} \subseteq d(I \Omega^j).$$
\end{lem}
\begin{proof}
The first equality is trivial.  Let $du \wedge dw \in dI \wedge d\Omega^{j-1}$.  Then
$u dw \in I \Omega^j$ and
$$
d(u \wedge dw) = dh \wedge dw + u d^2(w) = du \wedge dw.
$$
\end{proof}

\subsection{Principal ideals}
The ideal $I$ associated with the representation varieties of $\Sigma_{1,1}$ and $\Sigma_{0,4}$ are principal.  Hence it is an important case for us.  For this subsection, we assume $I$ to be principal.
Since $\C^n$ is acyclic and $\Omega^{n+1} = 0$, $d_n : \Omega^{n-1} \lto \Omega^n$ is onto; consequently, a form in $\Omega^{n-1}$ is closed if and only if it is exact.  This implies that a form $w \in \Omega_U^{n-1}$ is exact if and only if $d w = 0$ in $\Omega^n$.
This together with Remark~\ref{rem:duality} and Lemma~\ref{lem:small} extend
Figure~\ref{fig:3} to the following commutative diagram with exact rows:
\begin{figure}[H]
\begin{equation*}
\xymatrix{
d\Omega_U^{n-2}\ar@{^{(}->}[r] & \Omega_U^{n-1} \ar[r] & \H^{n-1} \ar@{=}[d] \ar[r] & 0\\
d\Omega^{n-2} + I \Omega^{n-1}  + dI \wedge \Omega^{n-2} \ar@{^{(}->}[r] \ar[d]^{d} \ar[u] & \Omega^{n-1} \ar[r]^p \ar[d]^{d} \ar[u] & \H^{n-1} \ar[r] \ar@{=}[d] & 0\\
d(I \Omega^{n-1}) \ar@{^{(}->}[r] & \Omega^n \ar[r]^q  & \H^{n-1} \ar[r] \ar[r] & 0.
}
\end{equation*}
\caption{}
\label{fig:principal}
\end{figure}
%

\section{Free groups and their $\SL(2,\C)$-representation varieties}
The fundamental group of the one-holed torus is a free group of two generators while those of the four-holed sphere and the two-holed torus are free groups of three generators.  The traces of elements in $\SL(2,\C)$ are $\SL(2,\C)$-conjugate invariant.  Therefore the moduli spaces $\M$ and $\M_C$ have trace coordinates.  Moreover conjugacy classes of $\SL(2,\C)$ are characterized by traces if we remove the identity class $\{\I\}$.  In this section, we introduce the trace coordinates for the free groups of two and three generators.  For a detailed and excellent exposition, see \cite{Go0}.

\subsection{The free group of two generators} \label{subsec:2}
Let $\F_2 = \langle F_1, F_2 \rangle$ be the free group of two generators.  For $\rho \in \Hom(\F_2, \SL(2,\C))$, let
$$
z_1 = \tr(\rho(F_1)), \ \ z_2 = \tr(\rho(F_2)), \ \  z_{12} = \tr(\rho(F_1 F_2)).
$$
Then the representation variety is $\M = \C^3$ with $\Os = \C[\z],$
where
$$
\z = \{z_1, z_2, z_{12}\}.
$$

\subsection{The free group of three generators}\label{subsec:3}
Let $\F_3 = \langle F_1, F_2, F_3 \rangle$ be the free group of three generators.  Let $\rho \in \Hom(\F_3, \SL(2,\C))$.  For $1 \le i < j < k \le 3$,
let
$$
z_i = \tr(\rho(F_i)), \ \ z_{ij} = \tr(\rho(F_i F_j)), \ \  z_{ijk} = \tr(\rho(F_i F_j F_k)).
$$
Then $\M$ is defined by the quotient $\Os = \C[\z]/(u)$ (See [\S5.1, \cite{Go0}]), where
\begin{eqnarray*}
\z & = & \{z_i, z_{ij}, z_{ijk} : 1 \le i < j < k \le 3\}, \\
u & = & 4-z_1^2-z_2^2-z_3^2-z_1 z_2 z_3 z_{123}-z_{123}^2+z_1 z_2 z_{12}+ \\
& & z_3 z_{123} z_{12}-z_{12}^2+z_1 z_3 z_{13}+z_2 z_{123} z_{13}-z_{13}^2+ \\
& &z_2 z_3 z_{23}+z_1 z_{123} z_{23}-z_{12} z_{13} z_{23}-z_{23}^2.
\end{eqnarray*}

\section{The representation varieties of the one-holed torus}\label{sec:1-hole}
This section studies the representation varieties of the one-holed torus with structure group $\SL(2,\C)$ and describes the Gau\ss-Manin connection on a natural family.
Let $g = m = 1$.  Then the fundamental group $\pi$ is isomorphic to $\F_2$, the free group of two generators \cite{Go0}.  We begin with renaming the variables in Section~\ref{subsec:2}:  Let $\x = \{x_1, x_2, x_3\}$ such that
$$
x_1 = z_1, \ \ x_2 = z_2, \ \ x_3 = z_{12}.
$$

With respect to the two generators, the boundary element is \cite{Go0}
$$T = F_1 F_2 F_1^{-1} F_2^{-1}.$$  Let $\rho \in \Hom(\pi, \SL(2,\C))$ and $t = \tr(\rho(T))$.  Then
\begin{equation}\label{eq:onehole}
t = t(\x) = -2+x_1^2+x_2^2-x_1 x_2 x_3+x_3^2 \in \C[\x].
\end{equation}

Following the notations of Section~\ref{sec:gen}, denote by $\Os \cong \C[\x]$ the coordinate ring of $\M \cong  \C^3$
and
by $(\Omega^\bullet, d)=(\Omega_\M^\bullet, d)$ its algebraic de Rham complex.
We have a morphism
$$
f_1 : \M \lto \Spec(\C[y]) \cong \C
$$
induced by the ring homomorphism
$$
f_1^* : \C[y] \lto \C[\x], \ \ \  f_1^*(y) = t.
$$
The representation varieties $\M_C$ are the fibres of $f$.
For a fixed $b \in \C$, the representation variety $\M_C$ is defined by $I_b = (t-b)$, i.e. $\M_C = \Spec(\Os/I_b)$.  We rename $\M_C$ as $\M_b$.
Let $$\psi_1(y) = y^2 - 4 \in \C[y].$$
\begin{prop}\label{rem:singular}
For a fixed $b \in \C$, $\M_b$ is singular if and only if $\psi_1(b) = 0$.
\end{prop}
\begin{proof}
Let $J(I_b)$ be the Jacobian ideal of $I_b$.
For the Gr\"obner basis computation for $J(I_b)$, we treat $b$ as a variable and use the elimination degree-lexicographic order on $\x$.  More specifically, we use the monomial order matrix
$$
W=
\left(
\begin{array}{cccc}
 1 & 1 & 1 & 0 \\
 1 & 0 & 0 & 0 \\
 0 & 1 & 0 & 0 \\
 0 & 0 & 0 & 1
\end{array}
\right).
$$
on
$$
\{x_3, x_2, x_1, b\}.
$$

Denote by $J_G$ the resulting Gr\"obner basis of $J(I_b)$.
The (constant) term in $J_G$ containing only $b$ is $\psi_1(b)$.  In other words, $\psi(b) \neq 0$ if and only if $J(I_b) = \Os$ if and only if $\M_b$ is smooth by Proposition~\ref{prop:singular}.
\end{proof}

\subsection{Computing de Rham $\H^2$}
We assume here that $\psi_1(b) \neq 0$ unless otherwise specified and
use the last row of Figure~\ref{fig:principal} to compute $\H^2$.

\begin{thm}\label{thm:1-hole}
$\H^2$ has dimension $h^2 = 5$ and a $\C$-basis
$$
B = \{1, x_1, x_2, x_3, x_1^2\}\otimes (x_1 dx_{23}).
$$
\end{thm}
These are parallel results to \cite{GN1}.
\begin{proof}
We shall use the last row of Figure~\ref{fig:principal} to show that $q(dB)$ is a basis for $\H^2$.  Notice that $W$ is degree-modified.  Denote also by $W$ the induced weight on $\Omega^3$, according to Definition~\ref{def:compatible weight}.

We first observe that $t \in \Os$ is symmetric and
$$
dt = (2x_1 - x_2 x_3) dx_1 + (2x_2 - x_3 x_1)dx_2 + (2x_3 - x_1 x_2)dx_3.
$$
Let $a = (a_1, a_2, a_3) \in \Z_{\ge 0}^3$.
\begin{lem}\label{lem:1}
Fix $i$ and set $a_i \ge 0$ and $a_j > 0$ for $j \neq i$.  Then $\x^a dx_{123} \sim v dx_{123}$ for some $v \in \Os$ with $\deg(v) < |a|$.
\end{lem}
\begin{proof}
Since $t$ is a symmetric polynomial, without the loss of generality, we may assume that $a_1 \ge 0$ and $a_j > 0$ for $j > 1$.
Let $s = a - (0,1,1)$ and
$$
w = d((t-b)\x^sdx_{23})  =  -((a_1+1) \x^a + v)dx_{123}.
$$
Then $w \in d (I_b \Omega^2)$ and
$\deg(v) < |a|$.  Hence $\x^a dx_{123} \sim -\frac{v}{a_1+1}dx_{123}$.  This also means $W(vdx_{123}) < W(\x^adx_{123})$ since $W$ is degree modified.
\end{proof}
\begin{lem}\label{lem:2}
Fix $i$ and set $a_i > 2$ and $a_j = 0$ for $j \neq i$.  Then $\x^a dx_{123} \sim vdx_{123}$ for some $v \in \Os$ with $W(vdx_{123}) < W(\x^adx_{123})$.
\end{lem}
\begin{proof}
Again since $t$ is a symmetric polynomial, we may assume that $i = 1$.
Let $a = (a_1,0,0)$ with $a_1 > 2$ and
\begin{equation}
w_1  =  d((t-b)(2x_1^{a_1-1}dx_{23}  + a_1 x_1^{a_1 - 2} x_2 dx_{13})) = (2\x^a + v)dx_{123}.
\end{equation}
Then $w_1 \in d (I_b\Omega^2)$.  Moreover, $v$ has the following properties: if $\x^s$ is a monomial in $v$, then either $|s| < a_1$ or $\x^s$ satisfies the hypothesis of Lemma~\ref{lem:1} with $|s| = a_1$.  In the latter case, $\x^sdx_{123}$ is cohomologous to a 3-form of degree less than $a_1$ by Lemma~\ref{lem:1}.  Hence in both cases, $\x^adx_{123}$ is cohomologous to a 3-form with a strictly lower weight.

Since $t$ is symmetric, similar arguments take care of the cases of $a =(0,a_2,0)$ and $a=(0,0,a_3)$ for $a_2, a_3 >2$ respectively, by permuting the indices of the items in $w_1$.
More specifically, for $a = (0,a_2,0)$, permute by $1 \leftrightarrow 2$ in the expression of $w_1$ to obtain:
\begin{equation} \label{eqn:w1}
w_2  =  d(-(t-b)(2x_2^{a_2-1}dx_{13}  + a_2 x_2^{a_2 - 2} x_1dx_{23})).
\end{equation}
For for $a = (0,0,a_3)$, permute by $1 \leftrightarrow 3$ in the expression of $w_1$ to obtain:
\begin{equation} \label{eqn:w2}
w_3 =  d((t-b)(2x_3^{a_3-1}dx_{12}  + a_3 x_3^{a_3 - 2} x_2 dx_{13})).
\end{equation}
\end{proof}

From the above two Lemmas, we may assume $a_i \le 2$ and $a_j = 0$ for $j \neq i$.
Suppose $a = (0,2,0)$.   Set $a_2 = 2$ for the expression of $w_2$ above, we get $w_2= 4(x_2^2 - x_1^2) dx_{123}\in d (I_b \Omega^3)$.  Hence $x_2^2 dx_{123}\sim x_1^2dx_{123}$.

Suppose $a = (0,0,2)$.  Set $a_3 = 2$ for the expression of $w_3$ above and obtain $w_3 = 4(x_3^2 - x_2^2)dx_{123}$.  Hence $x_3^2dx_{123} \sim x_2^2 dx_{123}$.

Hence we conclude that $q(dB)$ generates $\H^2$.  By Remark~\ref{rem:degree}, one needs to only check a finite (very few) number of at most cubic polynomial 2-forms to verify the linear independence of $q(dB)$.
\end{proof}

In this relatively simple situation, one can also compute the algebraic de Rham cohomologies for the two singular cases.

Suppose $b = -2$.  Then $\M_b$ has one singular point at origin and
$$
dI_b\wedge\Omega^2 + I_b \Omega^3 = \{x_1,x_2,x_3\}\otimes dx_{123}.
$$
This implies $\Omega_{\M_b}^3 = \C \otimes dx_{123}$.  Hence
$$x_1 dx_{23} \not\in \ker(d_2) \subseteq \Omega_{\M_b}^3.$$
Therefore
$$\{x_1, x_2, x_3, x_1^2\}\otimes x_1 dx_{23}$$
is a basis for $\H^2$ and $h_2 = 4$.

Suppose $b = 2$.  Then a Gr\"obner basis (with monomial order $W$) for $dI_b\wedge\Omega^2 + I_b \Omega^3$ is
$$
\{-4+x_1^2,x_1 x_2-2 x_3,-4+x_2^2,-2 x_2+x_1 x_3,-2 x_1+x_2 x_3,-4+x_3^2\} \otimes dx_{123}.
$$
This implies
$$\Omega_{\M_b}^3 = \C \otimes \{1, x_1, x_2, x_3\} \otimes dx_{123}.$$
Hence no 2-form of degree less than 3 is in $\ker(d_2)$.  Hence
$$\{x_1^2\} \otimes x_1 dx_{23}$$
is a basis for $\H^2$ and $h_2 = 1$.

Theorem~\ref{thm:Groth} only guarantees that the algebraic de Rham cohomologies agree with the smooth de Rham cohomologies.  These two results show that this is also true for these two particular singular spaces (compare \cite{GN1}).

\subsection{What the computer says}
Modern computer algebra has come of age and one may obtain much information directly from packages such as {\em Macaulay2} \cite{GS1}.
Denote by $\H_c^i$ the compact support cohomology of $\M_b$ and $h_c^i = \dim(\H_c^i)$.  Then
\begin{thm}
If $\M_b$ is smooth, then $$h_c^0 = 0, \ \ h_c^1 = 0, \ \ h_c^2 = 5, \ \ h_c^3 = 0, \ \ h_c^4 = 1.$$
\end{thm}
\begin{proof}
We have $\M_b \subseteq \C^3$ as a subvariety.  Since $\C^3$ is acyclic, Alexander duality gives \cite{Wa1}
$$
\H_c^i(\M_b)^* \cong \H_{dR}^{6-i-1}(\C^3 \setminus \M_b) \ \ \text{ for }  \ i \le 4, \ \ \
\H_c^5(\M_b) = 0.
$$
One then uses {\em Macaulay2} to compute $\H_{dR}^*(\C^3 \setminus \M_b)$ via the Oaku-Takayama algorithm to obtain the above numbers \cite{GS1, OT1}.
\end{proof}

The singularity of $\M_{-2}$ at $(0,0,0)$ is isolated.  Let $\B$ be a small $\epsilon$-ball of $(0,0,0)$.  Then for some $b \in f(\B)$ near but not equal to $-2$, $f_1^{-1}(b)\cap\B$ is homotopic to a bouquet of 2-spheres \cite{Mi1}.  One may apply Schultze's algorithm using the Brieskorn lattice method to compute the monodromy of $f_1|_\B$ around $-2$ \cite{Br1, Sch3}:
\begin{thm}\label{thm:Schultze algorithm}
$f_1^{-1}(b) \cap \B$ consists of one 2-sphere and the monodromy action is the $-1$ map.
\end{thm}
This means that if one goes around a small loop around $-2 \in \C$, the monodromy action on $f_1^{-1}(b)\cap \B$ is the antipodal map on the small sphere in $f_1^{-1}(b) \cap \B$ around $(0,0,0)$.  
Schultz implemented his algorithm in {\em Singular} \cite{DGPS}.  Notice also that this monodromy action does not arise from a Dehn twist action on $\Sigma_{1,1}$ because any Dehn twist induced monodromy action is the identity on the above $2$-sphere \cite{GN1}.

There is a natural compactification via the projectivization of $\M_b$ and much more information may be obtained from this projectivized object.
Consider the projective surface defined by the homogeneous polynomial
$$
\Psi(X_1,X_2,X_3,V) = V(X_1^2+X_2^2 + X_3^2) - X_1 X_2 X_3 - V^3(2+b).
$$
Then $\Psi$ defines the (projective) compactification $\bar{\M}_b \subseteq \P^3$.
A direct calculation shows that $\Psi$ is irreducible and that $\barM_b$ is smooth if and only if $\psi_1(b) \neq 0$.  We assume this is the case for the rest of this section.
Geometrically, $\bar{\M}_b \setminus \M_b$ consists of three copies of $\P^1$ defined by the equation $V=0$, pairwise intersecting at a point (with a total of three points of intersections).
{\em Macaulay2} gives us
\begin{thm}\label{thm:hodge}
The non-zero Hodge numbers of $\bar{\M_b}$ are
$$
h^{0,0}(\bar{\M_b}) = h^{2,2}(\bar{\M_b}) = 1, \ \ \  h^{1,1}(\bar{\M_b}) =7.
$$
\end{thm}
\begin{cor}\label{cor:hodge}
The Betti numbers of $\bar{\M_b}$ are
$$
h^0(\bar{\M_b}) = h^4(\bar{\M_b}) = 1, \ \ \  h^2(\bar{\M_b}) =7.
$$
\end{cor}
Notice that one may obtain the Betti numbers using the algebraic de Rham complex \cite{Wa1}.
\subsection{The groups $\H^0$ and $\H^1$}
The algebraic de Rham cohomology satisfies many of the usual cohomological axioms.  There is an excision sequence [Theorem 3.3, \cite{Ha1}]:
\begin{prop}[Excision]\label{prop:excision}
Suppose $\UU$ is smooth and $\VV \subseteq \UU$ a smooth subvariety of codimension $r$. Let $\WW = \UU \setminus \VV$. Then there is an exact sequence
\begin{equation*}
\xymatrix{
\cdots \ar[r] & \H^{i-2r}(\VV) \ar[r] & \H^i(\UU) \ar[r] & \H^i(\WW) \ar[r] & \H^{i-2r+1}(\VV) \ar[r] & \cdots\\
}
\end{equation*}
\end{prop}
\begin{cor}\label{cor:euler}
The Euler characteristics are additive:
$$
\chi(\UU) = \chi(\VV) + \chi(\WW).
$$
\end{cor}
\begin{cor}
If $\psi_1(b) \neq 0$, then
$h^0 = 1$ and $h^1 = 0$.
\end{cor}
\begin{proof}
The locus at infinity defined by $V = 0$ consists of three $\P^1$'s pairwise intersecting at one point.  Let $\VV$ be the disjoint union:
$$
\VV = \VV_1 \cup \VV_2 \cup \VV_3, \ \ \text{ where } \VV_1 = \P^1,  \VV_2= \C, \VV_3 = \C \setminus \{0\},
$$
each of which has codimension $1$ in $\bar{\M}_b$.
Let $\UU_0 = \bar{\M}_b$ and $\UU_{i+1} = \UU_i \setminus \VV_{i+1}$ for $0\le i \le 2$.
Notice that $\M_b = \UU_3$.
Proposition~\ref{prop:excision} gives the following exact sequence:
\begin{equation*}
\xymatrix{
0 \ar[r] & \H^0(\UU_i) \ar[r] & \H^0(\UU_{i+1}) \ar[r] & 0\\
}
\end{equation*}
Hence $h^0(\UU_i) = 1$ for $0 \le i \le 3$.

By Corollary~\ref{cor:euler},
$$
\chi(\bar{\M}_b) = \chi(\M_b) + \sum_{i=1}^3 \chi(\VV_i).
$$
By Corollary~\ref{cor:hodge}, $\chi(\bar{\M}_b) = 9$.  Hence
$$
\chi(\M_b) = \chi(\bar{\M}_b) - \sum_{i=1}^3 \chi(\VV_i) = 9 - (2 + 1 + 0) = 6.
$$
By Theorem~\ref{thm:1-hole},  $h^2(\M_b) = 5$.  Since $h^0(\M_b) = 1$,
$$
h^1(\M_b) = (5+1) - \chi(\M_b)  = 0.
$$
\end{proof}
\begin{rem}
All these calculations are done in the algebraic category and these results parallel those in \cite{GN1}.
\end{rem}

\subsection{The Gau\ss-Manin connection}
Recall the map $f_1 : \M \lto \C$ corresponding to the ring morphism
$$
f_1^* : \C[y] \lto \C[\x], \ \ \ y \mapsto t.
$$

$f_1$ is not smooth; however, it becomes smooth when the fibres over $\Spec(\C[y]/(\psi_1(y)))$ are removed.
Let $X$ and $Y$ be the respective localizations defined by
$$
\Os_X = \C[\x]_{f_1^*(\psi_1(y))}, \ \ \ \Os_Y = \C[y]_{\psi_1(y)}.
$$
Then $f_1 : X \lto Y$ is smooth.  The Gau\ss-Manin connection on $\mH^2$ is
$$
\nb : \mH^2 \lto \Omega_Y^1 \otimes_{f_1^*} \mH^2.
$$
Recall our choice of a basis for $\H^2$:
\begin{eqnarray*}
B & = & \{1, x_1, x_2, x_3, x_1^2\} \otimes x_1 dx_{23}\\
dB & = & \{1, 2x_1, x_2, x_3, 3x_3^2\}\otimes dx_{123}.
\end{eqnarray*}
For each $b \in \C$ with $\psi_1(b) \neq 0$, $P = (y - b)$ is a maximal prime of $\Os_Y$.  Let
$
\Phi_1 : Y_P \lto Y
$
be the localization map.
By Proposition~\ref{prop:localization},
$$
\Phi_1^*(\mH^2) \stackrel{\cong}{\lto} R^2 (f_1 \circ \Phi_1)_*(\Omega_{\Phi_1^*(X)/Y_P}).
$$
Let $\phi_1 : \Spec(k(P)) \lto Y_P.$
Then the following diagram
\begin{equation}
\xymatrix{
\Omega_{Y_P}^i \ar[r]^d \ar[d]^{\phi_1^*} & \Omega_{Y_P}^{i+1} \ar[d]^{\phi_1^*} \\
\Omega_{\M_b}^i \ar[r]^d & \Omega_{\M_b}^{i+1}
}
\end{equation}
commutes.  By Nakayama lemma, $B$ generates $\Phi_1^*(\mH^2)$.
Since this is true for every maximal $P$, by the local to global principle [Corollary 2.9, \cite{Ei1}],
$B$ generates $\mH^2$.  Hence $B$ serves as a basis for $\mH^2$.

Let $u \in \mH^2$.
Then $du$ is of the form $w \wedge dt$ for some $w \in \Omega_X^2$ and $\nb(u) = w \otimes_{f_1^*} dy$.  With the global basis $B$, we may write
$$\nb = d + E(t)\otimes_{f_1^*} dy,$$
where $d$ is the exterior differential operator of $(\Omega_Y^\bullet, d)$.

We now factor each element in $dB$ as a product of $w \wedge dt$ and write $w$ as a linear combination of basis elements in $B$.  This is carried out with the help of {\em Macaulay2} \cite{GS1}.  Let
$$
\eta = \frac{(t-2) x_1 dx_{23} +  x_3 (x_3^2-4)dx_{12} + (2 x_1 x_3 + 2x_2 - t x_2 - x_2 x_3^2)dx_{13}}{2 (t^2-4)}.
$$
Then $dx_{123} = \eta \wedge dt$.  Hence
$$
\left\{
\begin{array}{ccc}
d (x_1 dx_{23})  & = & \eta \wedge dt \\
d (x_1^2 dx_{23})  & = & 2 x_1 \eta \wedge dt \\
d (x_1^3 dx_{23})  & = & 3 x_1^2 \eta \wedge dt \\
d (x_1 x_2 dx_{23})  & = & x_2 \eta \wedge dt \\
d (x_1 x_3 dx_{23})  & = & x_3 \eta \wedge dt.
\end{array}
\right.
$$
We need to write the elements in the set $$D = \{1, 2x_1, 3x_1^2, x_2, x_3\} \otimes \eta$$ as linear combinations of elements in $B$.  Let $b \in \C \setminus \{\pm 2\}$ and consider
$$D \subseteq \mH^2 \otimes_{f_1^*} (\Os_{Y} /(y-b)) \cong \H^2.$$
Now we apply the algorithm in the proof of Theorem~\ref{thm:1-hole} to $D$.  This results in
$$
\left\{
\begin{array}{ccc}
\eta  & \sim & (\frac{9 x_1-x_1^3}{6 (b+2)}+\frac{-3 x_1+x_1^3}{6 (b-2)})dx_{23}\\
2 x_1 \eta  & \sim & \frac{3 x_1^2}{2 (b-2)} dx_{23}\\
3x_1^2 \eta  & \sim & \frac{-6x_1+2x_1^3}{b-2}dx_{23}\\
x_2 \eta & \sim & \frac{3 x_1 x_2}{2 (b-2)} dx_{23}\\
x_3 \eta  & \sim & \frac{3 x_1 x_3}{2 (b-2)} dx_{23}.
\end{array}
\right.
$$

Since the Jacobson radical of $\Os_Y$ is $\{0\}$, $E(t)$ with respect to the basis $B$ is
$$
E(t) =
\left(
\begin{array}{ccccc}
 \frac{3}{2 (t+2)}+\frac{-1}{2 (t-2)} & 0 & 0 & 0 & \frac{-1}{6 (t+2)}+\frac{1}{6 (t-2)}\\
 0 & \frac{3}{2(t-2)}  & 0 & 0 & 0\\
 0 & 0 &  \frac{3}{2(t-2)} & 0 & 0\\
 0 & 0 & 0 &  \frac{3}{2(t-2)} & 0\\
  \frac{-6}{t-2} & 0 & 0 &  0 & \frac{2}{t-2}
\end{array}
\right).
$$
Notice the $t+2$ and $t-2$ terms in the denominators.  These are the singular values around which the monodromy of $\nb$ is not trivial.
From $E(t)$, we see that $\nabla$ is a direct sum of three rank-1 systems and one rank-2 system.  Denote by $\D$ the Gau\ss-Manin connection for the rank-2 subsystem.
$Y = \C \setminus \{-2, 2\}$ is the three-holed sphere and
$$
\D = d + (\frac{A_2}{y-2} +  \frac{A_{-2}}{y+2}) dy,
$$
where
$$
A_2 = \left(
\begin{array}{cc}
 -\frac{1}{2} &  \frac{1}{6}\\
   -6 & 2
\end{array}
\right), \ \ \
A_{-2} =
\left(
\begin{array}{cc}
 \frac{3}{2} & -\frac{1}{6}\\
   0 & 0
\end{array}
\right).
$$
The exponential matrix at infinity is then
$$
A_\infty = - (A_2 + A_{-2}) =
\left(
\begin{array}{cc}
-1 & 0\\
 6 & -2
\end{array}
\right).
$$
The eigenvalues of $A_2$ and $A_{-2}$ are $0$ and $\frac{3}{2}$.  The eigenvalues of $A_\infty$ are $-1$ and $-2$.  Since the difference of the two eigenvalues of $A_\infty$ is a non-zero integer, one must take special care to compute the monodromy at $\infty$.  We make a change of variable $y \to \frac{1}{z}$.  Then
$$
\D = d + \frac{\sA(z)}{z} dz, \ \ \ \sA(z) = -(\frac{A_2}{(1-2z)} +  \frac{A_{-2}}{(1+2z)}).
$$
Now we follow [\S 6, \cite{Ma1}] to compute the monodromy at $z=0$.  First we compute the Taylor series of $\sA$ at $z=0$.  This gives us $\sA(0) = A_\infty$ and
$$
\frac{d \sA}{dz}(0) = 2A_{-2} - 2A_2 =
\left(
\begin{array}{cc}
4 & -\frac{2}{3}\\
12 & -4
\end{array}
\right).
$$
Following [\S 6, \cite{Ma1}], we set
$$
\varphi =
\left(
\begin{array}{cc}
-2 & -\frac{2}{3}\\
 0 & -2
\end{array}
\right).
$$
Then the monodromy at $z = 0$ is
$$
\N_\infty = e^{-2 \pi i \varphi} =
\left(
\begin{array}{cc}
1 & \frac{4 \pi i}{3}\\
 0 & 1
\end{array}
\right).
$$
The classical result of Riemann says that the global monodromy is determined by the local ones at the three punctures \cite{Ka1}.
A direct computation then shows that the monodromy group of this rank-2 subsystem is generated by the following elements
$$
\N_{-2} = \left(
\begin{array}{cc}
1 & 0 \\
 0 & -1
\end{array}
\right), \ \ \
\N_2 =
\left(
\begin{array}{cc}
1 & -\frac{4 \pi i}{3} \\
 0 & -1
\end{array}
\right), \ \ \
\N_\infty =
\left(
\begin{array}{cc}
1 & \frac{4 \pi i}{3}\\
 0 & 1
\end{array}
\right).
$$


Notice that $\M_b$ is not projective.
Moreover the locus at infinity of $\bar{\M}_b$ consists of three copies of $\P^1$, pairwise intersecting at one point.
The long exact sequence in Proposition~\ref{prop:excision}  then shows that the map $\iota^* : \H^2(\bar{\M}_b) \lto \H^2(\M_b)$ has rank equal to 4, where $\iota : \M_b \lto \bar{\M}_b$ is the inclusion.

\section{The representation varieties of the four-holed sphere}\label{sec:4-hole}
This section computes the cohomologies $\H^\bullet$ of the smooth $\SL(2,\C)$-representation varieties of a four-holed sphere.  Let $g = 0, m = 4$.
Then the fundamental group $\pi$ is isomorphic to $\F_3$, the free group of three generators \cite{Go0}.  Again we rename the variables in Section~\ref{subsec:3}:
Let $\x = \{x_1, x_2, x_3 \}$ such that
$$
x_1 = z_{12}, \ \ x_2 = z_{13}, \ \ x_3 = z_{23}.
$$

The four punctures correspond to
$$F_1, F_2, F_3, F_4 := (F_1 F_2 F_3)^{-1}.$$
For
$\rho \in \Hom(\F_3,G)$, let $\bt = \{t_1, t_2, t_3, t_4\}$ with
$$
t_i = \tr(\rho(F_i)), \ \ 1 \le i \le 4.
$$
With these new notations, let
\begin{eqnarray*}
u_4 = u_4(\bt) & = & 4-t_1^2-t_2^2-t_3^2-t_1 t_2 t_3 t_4-t_4^2+t_1 t_2 x_1+ \\
& & t_3 t_4 x_1-x_1^2+t_1 t_3 x_2+t_2 t_4 x_2-x_2^2+ \\
& &t_2 t_3 x_3+t_1 t_4 x_3-x_1 x_2 x_3-x_3^2.
\end{eqnarray*}
Then $\M = \Spec(\C[\x,\bt]/(u_4))$.  Let $\y = \{y_1,y_2,y_3,y_4\}$.   Then we have a morphism
$$
f_4 : \M \lto \Spec(\C[\y])
$$
induced by the ring homomorphism
$$
f_4^* : \C[\y] \lto \C[\x, \bt]/(u_4), \ \ \  f_4^*(y_i) = t_i.
$$

For a fixed element $\bb = (b_1, b_2, b_3, b_4) \in \C^4$, representing the (fixed) monodromies at the punctures, we rename $\M_C$ as $\M_\bb$.  Then $\M_\bb$ is defined by the ideal
$$I_\bb = (t_1 - b_1, t_2 - b_2, t_3 - b_3, t_4 - b_4, u_4) \subseteq \C[\x,\bt].$$
\begin{rem}
$\M_\bb$ is a subvariety of $\C^3$ defined by the principal ideal $$L_\bb := (u_4(\bb)) = I_\bb \cap \C[\x] \subseteq \C[\x]$$ for a fixed $\bb \in \C^4$.  Let $\Os = \C[\x]$.
\end{rem}
Introducing the symmetric coordinates, let $$\s = \{s_1, s_2, s_3, s_4\}$$ be the elementary symmetric polynomials in $\C[\y]$, i.e.
$$
s_i = \sum_{|a|=i, a_j \le 1} \y^a, \ \ \ 1 \le i \le 4.
$$
Let
\begin{eqnarray*}
\Delta(\y) & = &  (y_1^4+y_2^4+y_3^4+y_4^4)  -2 (y_1^2 y_2^2 + y_1^2 y_3^2+ y_1^2 y_4^2+y_2^2 y_3^2+ y_2^2 y_4^2)+8 y_1 y_2 y_3 y_4 +\\
& & (y_1^2 y_2^2 y_3^2 +y_1^2 y_3^2 y_4^2+y_2^2 y_3^2 y_4^2)-(y_1^3 y_2 y_3 y_4 +y_1 y_2^3 y_3 y_4 + y_1 y_2 y_3^3 y_4+y_1 y_2 y_3 y_4^3)\\
& = & s_1^4 - (4 s_1^2 s_2 + s_1^2 s_4 + s_4 s_1^2) + (8s_1 s_3 + s_3^2).
\end{eqnarray*}
Let $\psi_4(\y) \in \C[\s] \subseteq \C[\y]$ be the symmetric polynomial
$$
\psi_4(\y) = \Delta(\y)^2 \prod_{i=1}^4 (y_i^2-4).
$$
\begin{thm}
The singularity locus is defined by the symmetric polynomial $\psi_4$.
This is to say that $\M_\bb$ is singular if and only if $\psi_4(\bb) = 0$.
\end{thm}
\begin{proof}
Let $J(L_\bb)$ be the Jacobian ideal of $L_\bb$.
For the Gr\"obner basis computation, we treat $\bb$ as variables and use monomial order
$$
W=
\left(
\begin{array}{ccccccc}
 1 & 1 & 1 & 0 & 0 & 0 & 0 \\
 1 & 0 & 0 & 0 & 0 & 0 & 0 \\
 0 & 1 & 0 & 0 & 0 & 0 & 0 \\
 0 & 0 & 0 & 1 & 1 & 1 & 1 \\
 0 & 0 & 0 & 1 & 1 & 1 & 0 \\
 0 & 0 & 0 & 1 & 1 & 0 & 0 \\
 0 & 0 & 0 & 1 & 0 & 0 & 0
\end{array}
\right)
$$
on variables
$$
\{x_3, x_2, x_1, b_4, b_3, b_2, b_1\}.
$$
Denote by $J_G$ the resulting Gr\"obner basis of $J(L_\bb)$.
The (constant) term in $J_G$ that contains only $\bb$ is $\psi_4(\bb)$.  In other words, $\psi_4(\bb) \neq 0$ if and only if $J(L_\bb) = \Os$ if and only if $\M_\bb$ is smooth by Proposition~\ref{prop:singular}.
\end{proof}

The $\SL(2,\C)$-representation variety of the 4-holed sphere is of importance.  As far as the author is aware, this is the first explicit computation of the singularity locus.  The factor $\prod_{i=1}^4 (b_i^2-4)$ corresponds to the representation varieties of the three-holed sphere.
The three-fold defined by $\Delta$ is worthy of further analysis.

\subsection{Computing $\H_{dR}^\bullet$}

\begin{thm}\label{thm:4-hole}
If $\M_\bb$ is smooth, then
$\H^2$ has dimension $h^2 = 5$ and a $\C$-basis
$$
B = \{1, x_1, x_2, x_3, x_1^2\}\otimes (x_1 dx_{23}).
$$
\end{thm}
\begin{proof}
All (smooth) $\M_\bb$ have isomorphic $\H^\bullet$ for $\psi_4(\bb) \neq 0$.
Let $\bb = (1,0,0,0)$.  Then $\psi_4(\bb) = 1\neq0$.  Hence $f_4^{-1}(\bb)$ is a smooth fibre.  

Recall the function $f_1$ and the objects $b$ and $t(\x)$ from Section~\ref{sec:1-hole}.  Let $b=1$, make a change of coordinates $\x \to -\x$ and consider the ideal $(t(-\x) - b) \subseteq \Os$.  Then a direct calculation shows
$$I_b \cong (t(-\x) - b) = L_\bb$$
as ideals of $\Os$.  
Hence $\M_\bb \cong \M_b$.  The result then follows from Theorem~\ref{thm:1-hole} (Compare \cite{GN1}).
\end{proof}
\begin{rem}
One can similarly prove that $h^0=1$ and $h^1 = 0$.
\end{rem}

\subsection{The limit of computer power}

Recall the morphism $f_4 : \M \lto \C^4$ corresponding to the ring morphism
$$
f_4^* : \C[\y] \lto \C[\x,\bt], \ \ \ y_i \mapsto t_i.
$$
$f_4$ is not smooth; however, it becomes smooth when the fibres over $\Spec(\C[\y]/(\psi_4(\y)))$ are removed.
Let $X$ and $Y$ be the respective localizations defined by
$$
\Os_X = \C[\x,\bt]_{f_4^*(\psi_4(\y))}, \ \ \ \Os_Y = \C[\y]_{\psi_4(\y)}.
$$
Then $f_4 : X \lto Y$ is smooth.  The Gau\ss-Manin connection on $\H^2$ is
$$
\nb : \mH^2 \lto \Omega_Y^1 \otimes_{f_4^*} \mH^2.
$$
Again as in the case of 1-holed torus, the fibre over $\bb \in Y$ of $\mH^2$ is isomorphic to $\H^2$ as a $\C$-vector space and generated by $B$.  In theory, one then follows the method of Section ~\ref{sec:1-hole} and factors $du$ as
$$
du = \sum_{i=1}^4 w_i \wedge dt_i
$$
for $u \in B$.
We then have
$$
\nb (u) = du = \sum_{i=1}^4 w_i \otimes_{f_4^*} dy_i
$$
for each $u \in B$ to obtain the connection matrices $E_i(\bt)$ for $1 \le i \le 4$ so that
$$
\nb = d + \sum_{i=1}^4 E_i(\bt) \otimes_{f_4^*} dy_i.
$$
Unfortunately, the computation involved in this factorization overwhelmed the computers in our possession and is likely to overwhelm any currently available computers.


\section{The representation varieties of the two-holed torus}\label{sec:2-hole}
Let $g = 1, m = 2$.
Then the fundamental group $\pi$ is again isomorphic to $\F_3$, the free group of three generators \cite{Go0}.
Change $\z$ to $\x$ as before,
$$x_i = z_i, \ \ 1 \le i \le 3; \ \ x_{ij} = z_{ij}, \ \ \ 1 \le i < j \le 3,$$
$$\x = \{x_i, \ \ 1 \le i \le 3; \ \ x_{ij}, \ \ \ 1 \le i < j \le 3\}, \ \ \ \y = \{y_1, y_2\}$$

The two punctures correspond to
$F_1 F_2 F_3$ and $F_1 F_3 F_2$, respectively.
For
$\rho \in \Hom(\F_3,\SL(2,\C))$, let $\bt = \{t_1, t_2\}$ with
$$t_1 = z_{123} = \tr(\rho(F_1 F_2 F_3)), \ \ \ t_2  = \tr(\rho(F_1 F_3 F_2))$$
i.e $\bt$ represents the holonomies at the two punctures.
Let
\begin{eqnarray*}
u_2 & = & 4-x_1^2-x_2^2-x_3^2-x_1 x_2 x_3 t_1-t_1^2+x_1 x_2 x_{12}+ \\
& & x_3 t_1 x_{12}-x_{12}^2+x_1 x_3 x_{13}+x_2 t_1 x_{13}-x_{13}^2+ \\
& &x_2 x_3 x_{23}+x_1 t_1 x_{23}-x_{12} x_{13} x_{23}-x_{23}^2\\
u_s & = & t_1 + t_2 - (x_3 x_{12} + x_2 x_{13} + x_1 x_{23} - x_1 x_2 x_3 ).
\end{eqnarray*}

Then $\M = \Spec(\C[\x,\bt]/(u_2,u_s))$ and we have a morphism
$$
f_2 : \M \lto \Spec(\C[\y])
$$
induced by the ring homomorphism
$$
f_2^* : \C[\y] \lto \C[\x, \bt]/(u_2, u_s), \ \ \  f_2^*(y_1) = t_1, \ \ f_2^*(y_2) = t_2.
$$
For a fixed element $\bb \in \C^2$, representing the (fixed) monodromies at the punctures, we rename $\M_C$ as $\M_\bb$.  Then $\M_\bb \subset \C^8$ is defined by the ideal
$$I_\bb = (t_1 - b_1, t_2 - b_2, u_2, u_s) \subseteq \C[\x,\bt].$$
For an elaborate exposition of the above calculations, see \cite{Go0}.

Let $\psi_2(\y) \in \C[\y]$ be the symmetric polynomial
$$
\psi_2(\y) =  (y_1^2-4) (y_2^2-4) (y_1 - y_2)^2.
$$
\begin{thm}
For fixed $\bb \in \C^2$, $\M_\bb$ is singular if and only if $\psi_2(\bb) = 0$.
\end{thm}
\begin{proof}
Let $J(I_\bb)$ be the Jacobian ideal of $I_\bb$.
For the Gr\"obner basis computation, we treat $\bb$ as variables and use the monomial order
$$
W =
\left(
\begin{array}{ccccccccc}
 1 & 1 & 1 & 1 & 1 & 1 & 0 & 0 \\
 1 & 0 & 0 & 0 & 0 & 0 & 0 & 0 \\
 0 & 1 & 0 & 0 & 0 & 0 & 0 & 0 \\
 0 & 0 & 1 & 0 & 0 & 0 & 0 & 0 \\
 0 & 0 & 0 & 1 & 0 & 0 & 0 & 0 \\
 0 & 0 & 0 & 0 & 1 & 0 & 0 & 0 \\
 0 & 0 & 0 & 0 & 0 & 0 & 1 & 1 \\
 0 & 0 & 0 & 0 & 0 & 0 & 1 & 0
\end{array}
\right).
$$
on variables
$$
\{x_{23}, x_{13}, x_{12},x_3, x_2, x_1, b_2, b_1\}.
$$
Denote by $J_G$ the resulting Gr\"obner basis of $J(I_\bb)$.
The (constant) term in $J_G$ that contains only $\bb$ is $\psi_2(\bb)$.  In other words, $\psi_2(\bb) \neq 0$ if and only if $J(I_\bb) = \Os$ if and only if $\M_\bb$ is smooth by Proposition~\ref{prop:singular}.
\end{proof}
In this case, there is no mystery of the singular locus.  It contains the representation varieties of the on-holed torus.

\end{document}